\documentclass[12pt,reqno]{amsart}
\usepackage[all]{xy}
\usepackage{amssymb}
\usepackage{amsthm}
\usepackage{cite}
\usepackage{hyperref}
\usepackage[normalem]{ulem}
\usepackage{amsmath,mathtools}
\usepackage{amscd,enumitem}
\usepackage{graphicx}
\usepackage{caption}
\usepackage[justification=centering]{caption}
\usepackage{verbatim}
\usepackage{eurosym}
\usepackage{float}
\usepackage{color}
\usepackage{dcolumn}
\usepackage[mathscr]{eucal}
\usepackage[all]{xy}
\usepackage{bbm}
\usepackage{ytableau}
\usepackage[textheight=8.7in, textwidth=6.6in]{geometry}
\newtheorem*{conj*}{Conjecture}
\newtheorem{theorem}{Theorem}[section]

\theoremstyle{definition}

\theoremstyle{plain}
\newtheorem{cor}[theorem]{Corollary}
\newtheorem{lemma}[theorem]{Lemma}

\newtheorem{prop}[theorem]{Proposition}

\newtheorem*{rmk}{Remark}

\newtheorem*{dfn}{Definition}

\newcommand{\Z}{\mathbb{Z}}

\newcommand{\R}{\mathbb{R}}

\newcommand{\C}{\mathbb{C}}

\numberwithin{equation}{section}

\newtheoremstyle{example}
  {\topsep}   % ABOVESPACE
  {\topsep}   % BELOWSPACE
  {\normalfont}  % BODYFONT
  {0pt}       % INDENT (empty value is the same as 0pt)
  {\bfseries} % HEADFONT
  {.}         % HEADPUNCT
  {5pt plus 1pt minus 1pt} % HEADSPACE
  {}          % CUSTOM-HEAD-SPEC
\theoremstyle{example}

\def\({\left(}
\def\){\right)}

\usepackage{centernot}
\title{Annular Links from Thompson's Group $T$}
\author{Louisa Liles}

\address{Department of Mathematics, University of Virginia, Charlottesville, VA 22904}
\email{lml2tb@virginia.edu}

\thanks{The author acknowledges partial suport from the NSF RTG Grant DMS-1839968 and NSF Grant DMS-2105467}
\begin{document}
\maketitle 
\begin{abstract} In 2014 Jones showed how to associate links in the $3$-sphere to elements of Thompson's group $F$. We provide an analogue of this program for annular links and Thompson's group $T$.
The main result is that any edge-signed graph embedded in the annulus is the Tait graph of an annular link built from an element of $T$. In analogy to the work of Aiello and Conti \cite{aielloconti}, we also show that the coefficients of certain unitary representations of $T$ recover the Jones polynomial of annular links.
\end{abstract} 
\section{Introduction and Statement of Main Results}
Vaughan Jones introduced a method of constructing links in the $3$-sphere from elements of the Thompson group $F$, which are piecewise linear orientation-preserving self-homeomorphisms of the unit interval \cite{jones14,jones18}. Jones proved that the Thompson group $F$ gives rise to all link types in the $3$-sphere, suggesting that it can be used as an analogue of braid groups for producing links \cite[Theorem $5.3.1$]{jones14}. 

This paper provides a method for building links in the thickened annulus $\mathbb{A} \times I$ from Thompson's group $T$, which contains $F$ and whose elements are piecewise-linear orientation-preserving self-homeomorphisms of $S^{1}$. This method recovers Jones' construction for the subgroup $F$.

Given $g \in T$, one can follow the processs introduced in Section \ref{sec: construction} to build an annular link $\mathcal{L}_{\mathbb{A}}(g)$. On the other hand, given an edge-signed graph 
$\Gamma \hookrightarrow \mathbb{A}$, one can construct a diagram of an annular link $L_{\mathbb{A}}(g)$ in analogy with Tait's construction of links from planar graphs; see Figure \ref{fig: tait graph ex}. Jones proved that given any Tait graph $\Gamma \in \mathbb{R}^{2}$, there is some $g \in F$ which produces the same link as $\Gamma$. The following theorem states that the same is true for annular links and $T$:  
\begin{theorem}\label{thm: main}
    Let $\Gamma \hookrightarrow \mathbb{A}$ be an edge-signed embedded graph. Then there exists some $g \in T$ such that $\mathcal{L}_{\mathbb{A}}(g)$ is isotopic in $\mathbb{A} \times I$ to $L_{\mathbb{A}}(\Gamma)$.
\end{theorem} 

The construction of links in the $3$-sphere arose naturally in Jones' definition of certain unitary representations of $F$ and $T$ \cite{jones14}. The Kauffman bracket and Jones polynomial of links in the $3$-sphere were then shown to arise as coefficients of these unitary representations of $F$ \cite{aielloconti1,aiellocontijones} and T \cite{aielloconti}; this was accomplished by proving that they are functions of positive type. This paper establishes a similar result for annular links and $T$, which follows from the construction of links outlined in Section \ref{sec: construction}, and from \cite[Theorems $6.2$ and $7.4$]{aielloconti}. We now introduce some notation necessary to state the precise result. 

Elements of $T$ can be specified by triples $(R,S;k)$, where $R$ and $S$ are trees and $k$ is an integer; see Section \ref{sec: FT intro} for more details. When $R$, $S$ and $k$ are relevant, we will use  $\mathcal{L}_{\mathbb{A}}(R,S;k)$ to refer to the link resulting from the unique element $g \in T$ determined by $(R,S;k)$. Using this notation, we can establish the Jones polynomial of annular links as a function of positive type on $\vec{T}$, the oriented subgroup of $T$ which was first introduced by Jones in \cite{jones14} and is further discussed in Section \ref{sec: construction}.

\begin{cor}\label{thm: jones rep} For $g=(R,S;k) \in \vec{T}$, let $n$ be the number of leaves in $R$, and let $V_{\mathcal{L}}^{\mathbb{A}}(t)$ denote the Jones polynomial of an annular link $\mathcal{L}$, where unknotted curves wrapping once around $\mathbb{A}$ are equal to $(-t^{\frac{-1}{2}}-t^{\frac{1}{2}})$. Define $V_{g}^{\mathbb{A}}(t):\vec{T} \to \mathbb{C}$ analogously to \emph{\cite{aielloconti}}, that is,

\[V_{g}^{\mathbb{A}}(t):=V_{\mathcal{L}_{\mathbb{A}}(R,S;k)}^{\mathbb{A}}(t)(-t^{\frac{-1}{2}}-t^{\frac{1}{2}})^{-n+1}.\] Then, for $t \in \{1,i,e^{\pm \frac{\pi i}{3}}\}$, $V_{g}^{\mathbb{A}}(t)$ is a function of positive type on $\vec{T}$, and consequently the Jones polynomial of $\mathcal{L}_{\mathbb{A}}(g)$, evaluated at $t\in \{1,i,e^{\pm \frac{\pi i}{3}}\}$, arises as the coefficient of a unitary representation of $\vec{T}$.
\end{cor} 

\begin{figure}[t]
\begin{center}
\includegraphics[scale=0.2]{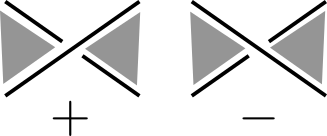}
\end{center}
\begin{center}
\includegraphics[scale=0.3]{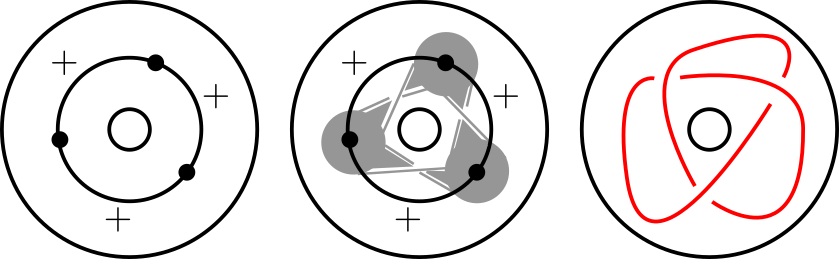}\put(-167,-15){$\Gamma$}\put(-45,-15){$L_{\mathbb{A}}(\Gamma)$}
\caption{An annular link built from an edge-signed graph embedded in $\mathbb{A}$.}\label{fig: tait graph ex}
\end{center}
\end{figure}

Annular links arise naturally in the study of knot theory, categorification \cite{PSA}, and representation theory of planar algebras\cite{graham1998representation,Jones01,JR,ghosh2011planar}. The interplay between the Thompson group and link theory is an emerging subject, and its full interaction with categorification, planar algebras, and representation theory is still being developed. Annular links will likely play an important role in this theory.

The paper proceeds as follows. Section \ref{sec: FT intro} provides an overview of Thompson's groups $F$ and $T$, and outlines Jones' construction of links in $S^{3}$ from $F$. Section \ref{sec: construction} introduces the construction of annular links from $T$ and connects it to Jones' unitary representations. In this section, the concept of \textit{Annular Thompson Badness} is presented as an extension of Jones' concept of Thompson Badness, which he uses to prove that $F$ can produce all link types. Section \ref{sec: proof} uses Annular Thompson Badness to prove Theorem \ref{thm: main}.

\section*{Acknowledgements}
The author is thankful for the guidance provided by her advisor Vyacheslav Krushkal.

\section{Thompson's Groups $F$ and $T$} \label{sec: FT intro}
\subsection{Thompson's Group $F$}\label{subsec: FT intro 1}

The Thompson group $F$ consists of piecewise linear orientation-preserving homeomorphisms of the unit interval $[0,1]$ such that all derivatives are powers of $2$ and all points of non-differentiability occur at dyadic numbers, that is, numbers of the form $\frac{a}{2^{b}}$ for $a, b \in \Z$. For example:
\[g(t)=\begin{cases} \frac{1}{2}t & 0 \leq t \leq \frac{1}{2}\\ t-\frac{1}{4} & \frac{1}{2} \leq t \leq \frac{3}{4} \\ 2t -1 & \frac{3}{4} \leq t \leq 1.\end{cases}\]

A \textit{standard dyadic partition} is a partition of the unit interval such that all intervals are of the form $[\frac{a}{2^{b}},\frac{a+1}{2^{b}}]$. Any ordered pair of standard dyadic partitions with the same number of parts determines an element of $F$, given by the function sending the first partition to the second. For example, the function $g$ above is given by the ordered pair \[\{[0,\frac{1}{2}],[\frac{1}{2}, \frac{3}{4}],[\frac{3}{4},1]\},\{[0,\frac{1}{4}],[\frac{1}{4}, \frac{1}{2}],[\frac{1}{2},1]\}).\] Standard dyadic partitions can be represented as planar, rooted, binary trees, where each leaf represents an interval of the partition.  Therefore, a pair of such trees also determines an element of $F$. This pair of trees is often represented by taking the vertical reflection of $S$ and attaching it to $R$ along their leaves; see Figure \ref{fig:reduced_g}. 

\begin{figure}[H]
\includegraphics[scale=0.35]{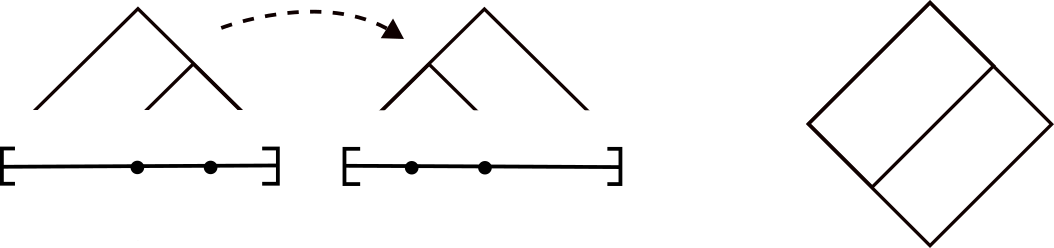}\put(-275,50){$R$}\put(-125,50){$S$}\put(-205,75){$g$}\put(-245,0){$\frac{1}{2}$}\put(-225,0){$\frac{3}{4}$}\put(-173,0){$\frac{1}{4}$}\put(-153,0){$\frac{1}{2}$}
\caption{A pair of standard dyadic partitions, their corresponding trees $R$ and $S$, and their associated element $g \in F$.}
\label{fig:reduced_g}
\end{figure}
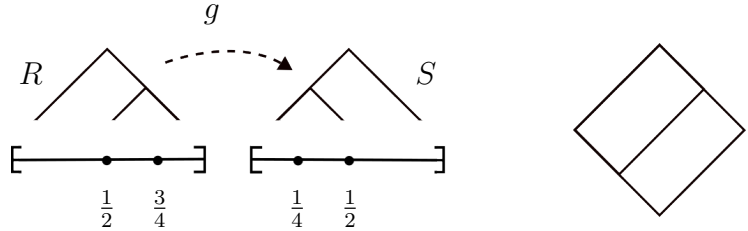

Conversely, for every $g \in F$ there is a standard dyadic partition $J$ such that $g(J)$ is standard dyadic. The pair $(J,g(J))$ therefore determines $g$, but this pair is not unique. For any refinement $J'$ of $J$ which also standard dyadic, $(J',g(J'))$ also represents $g$. In terms of trees, refining a pair of partitions corresponds to adding finitely many 
\textit{cancelling carets} to their pair of trees, as shown in Figure \ref{carets}.

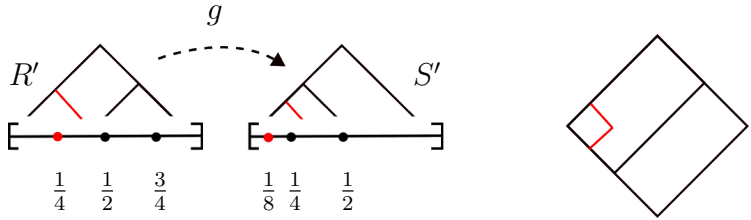
\begin{figure}[H]
\includegraphics[scale=0.35
]{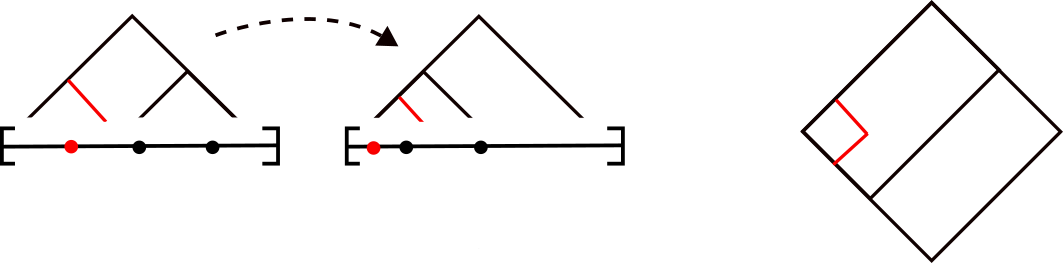}\put(-246,7){$\frac{1}{2}$}\put(-226,7){$\frac{3}{4}$}\put(-175,7){$\frac{1}{4}$}\put(-155,7){$\frac{1}{2}$}\put(-280,50){$R'$}\put(-127,50){$S'$}\put(-205,75){$g$}\put(-264,7){$\frac{1}{4}$}\put(-185,7){$\frac{1}{8}$}
\caption{A pair-of-trees representation of the same element $g$ from Figure \ref{fig:reduced_g}, which differs from the pair in Figure \ref{fig:reduced_g} by a cancelling caret.}
\label{carets}
\end{figure}
In fact, any two pairs of trees representing the same element of $F$ must differ by the addition or deletion of finitely many cancelling carets, and a pair of trees is called \textit{reduced} if no carets can be cancelled. Reduced pairs of planar, rooted, binary trees are therefore in bijection with elements of $F$; more details of this correspondence can be found in \cite{belk}. From now on, an ordered pair $(R,S)$ will refer to both a pair of standard dyadic partitions and its associated pair of trees, and elements of $F$ will be specified by these pairs.

Pairs of trees corresponding to elements of $F$ are part of a broader class of graphs called \textit{strand diagrams}, introduced by Belk in \cite{belk}. A general strand diagram can be \textit{reduced} according to moves of Type I and II, which were independently found by \cite{belk, gubasapir}. These moves are useful for visualizing the group operation in $F$. To compose $g$ with $f$, place the pair of trees for $g$ below that of $f$ as in Figure \ref{fig: comp in F} and then reduce the resulting strand diagram. This leads to the unique  reduced pair-of-trees diagram representing $g \circ f$. 
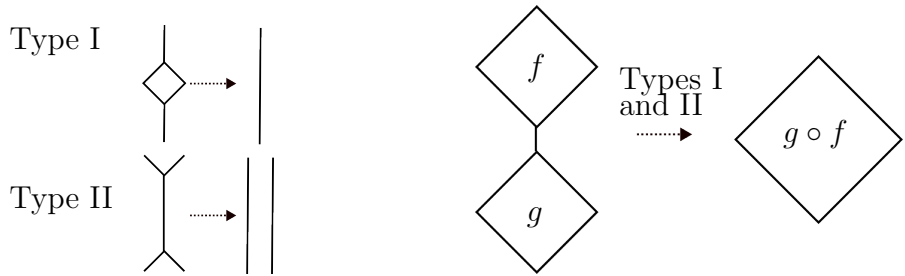
\begin{figure}[H]
\begin{center}
\includegraphics[scale=0.25] {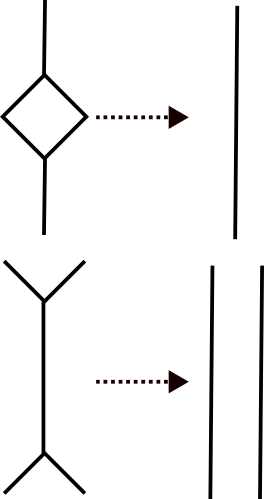}\put(-100,85){Type  I}\put(-100,25){Type II} \hspace{1in}
\includegraphics[scale=0.3]{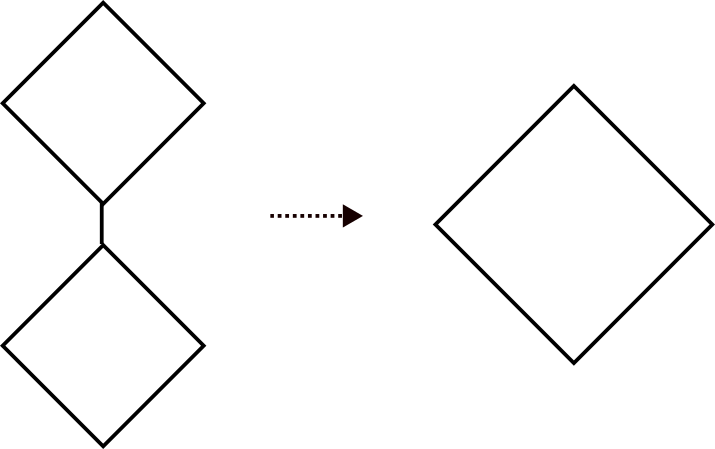}\put(-142,75){$f$}\put(-142,20){$g$}\put(-45,50){$g \circ f$}\put(-107,60){and II}\put(-107,70){Types I}
\caption{Diagrammatic composition in $F$ as given by \cite{belk,gubasapir}.}\label{fig: comp in F}\end{center}
\end{figure}

\subsection{Link Diagrams in the plane from $F$} Although Jones' original construction of links was formulated in terms of unitary representations of $F$, Jones provided two equivalent diagrammatic methods for building links\cite{jones14,jones18}. 

The first, pictured in Figure \ref{fig: hopf}, turns a reduced pair of trees $(R,S)$ into a reduced pair of ternary trees $(\phi(R),\phi(S))$, connects the two roots and the leaves from left to right, and then changes $4$-valent vertices to crossings.

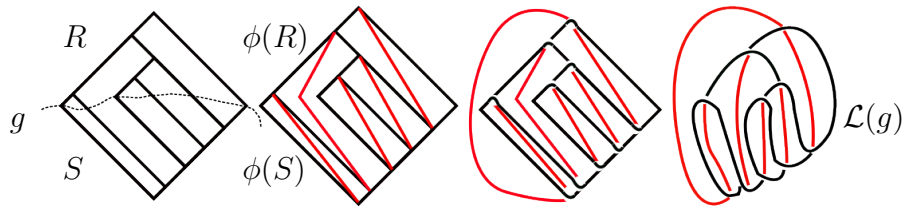
\begin{figure}[H]
\begin{center}
    \includegraphics[scale=0.4]{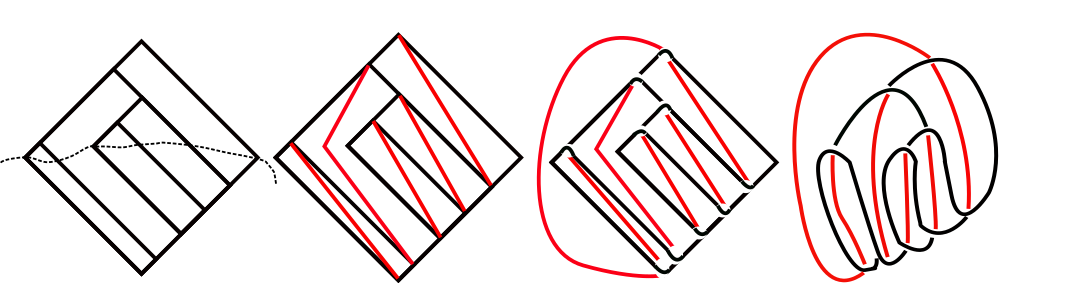}\put(-320,60){$R$}\put(-252,60){$\phi(R)$}\put(-252,10){$\phi(S)$}\put(-320,10){$S$}\put(-340,30){$g$}\put(-25,30){$\mathcal{L}(g)$}\caption{A Hopf link created from an element of $F$ via the construction introduced by Jones \cite{jones14}.}\label{fig: hopf}
\end{center}
\end{figure}

The second method, pictured in Figure \ref{fig: gamma RS}, builds the Tait graph $\Gamma(g)$ of $\mathcal{L}(g)$. For $g=(R,S)$, one makes two graphs $\Gamma(R)$ and $\Gamma(S)$ which have the same number of edges. Specifically, $\Gamma(R)$ has one vertex for each leaf of $R$, and it is placed to immediately to the left of the leaf. The vertices for $\Gamma(S)$ are created in the same way from $S$. For every edge $e$ in $R$ (resp. $S$) that slopes up and to the right, $\Gamma(R)$ (resp. $\Gamma(S)$) will have one edge which trasverlsely intesects $e$ once and no other edges. $\Gamma(g)$ is then built by reflecting $\Gamma(S)$ over the $x$-axis and identifying its edges with those of $\Gamma(R)$. Edges of $\Gamma(g)$ originating from $\Gamma(R)$ are given a positive sign and edges originating from $\Gamma(S)$ are given a negative sign.
\begin{center}
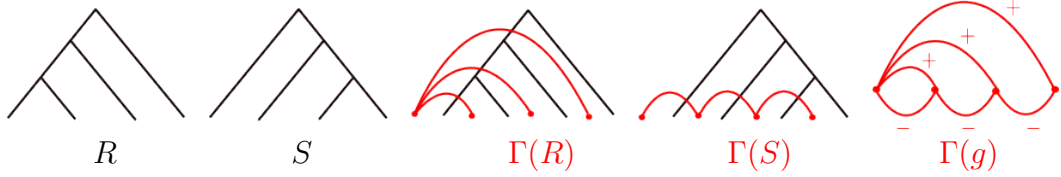
\begin{figure}[H]
\includegraphics[scale=0.6]{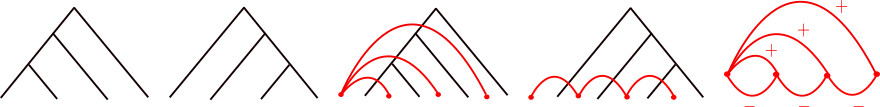}\put(-365,-13){$R$}\put(-45,-13){\color{red}$\Gamma(g)$}\put(-125,-13){\color{red}$\Gamma(S)$}\put(-208,-13){\color{red}$\Gamma(R)$}\put(-290,-13){$S$}
\caption{$\Gamma(g)$, the Tait graph for $\mathcal{L}(g)$, where $g$ is specified by $(R,S)$.}\label{fig: gamma RS}
\end{figure}
\end{center} 
\subsection{Thompson Badness}
To detect whether a general edge-signed planar graph $\Gamma$ is equal to $\Gamma(g)$ for some $g \in F$, Jones introduced \textit{Thompson Badness}, a quantity which is zero exactly when $\Gamma=\Gamma(g)$. To calculate Thompson Badness, first embed $\Gamma \hookrightarrow \R^{2}$ such that all vertices are on the $x$ axis, the leftmost vertex is at the origin, and for each edge, its interior is either entirely above or entirely below the $x$ axis. Consider each edge to be oriented from left to right, so that its rightmost vertex is considered the \textit{terminal vertex.} The formula for Thompson Badness, which will be given momentarily, depends on the cardinality of the following sets:
\begin{align*}
e^{in}_{v}&:=\{e \in e(\Gamma):  v\text{ is the terminal vertex of }e\}\\
e^{up}&:=\{e \in e(\Gamma): \text{int}(e)\text{ is in the upper half-plane}\}\\
e^{down}&:=\{e \in e(\Gamma): \text{int}(e)\text{ is in the lower half-plane}\}\\
e^{up}_{-}&:=\{e \in e(\Gamma): \text{int}(e)\text{ is in the upper half-plane and $e$ has sign $-$}\}\\
e^{down}_{+}&:=\{e \in e(\Gamma): \text{int}(e)\text{ is in the lower half-plane and $e$ has sign $+$}\}.\\
\end{align*}
Jones defines Thompson Badness as \[TB(\Gamma)=\sum_{v\in V(\Gamma)\setminus \{(0,0)\}}(|1-|e^{in}_{v} \cap e^{up}||+|1-|e^{in}_{v}\cap e^{down}||)+|e^{up}_{-}|+|e^{down}_{+}|\] and shows that $TB(\Gamma)=0$ if and only if $\Gamma = \Gamma(g)$ for some $g \in F$ \cite[Sections $4$ and $5$]{jones14}. 
\subsection{The oriented subgroup $\vec{F}$} Jones defined $\vec{F}$ as the set of elements $g \in F$ whose link diagram $\mathcal{L}(g)$, when given the checkerboard shading, results in an orientable surface, i.e. a Seifert surface for $\mathcal{L}(g)$.
Equivalently, this can be expressed in terms of the chromatic polynomial $Chr_{\Gamma(g)}(Q)$: \[\vec{F}=\{g \in F | Chr_{\Gamma(g)}(2)=2\}.\]

If one follows the convention that the leftmost face of the checkerboard surface is always positively oriented, each $g \in \vec{F}$ builds a link $L(g)$ with a natural orientation, namely that induced by the orientation of the checkerboard surface as in Figure \ref{fig: checkerboard}.
\begin{figure}[h]
\includegraphics[scale=0.35]{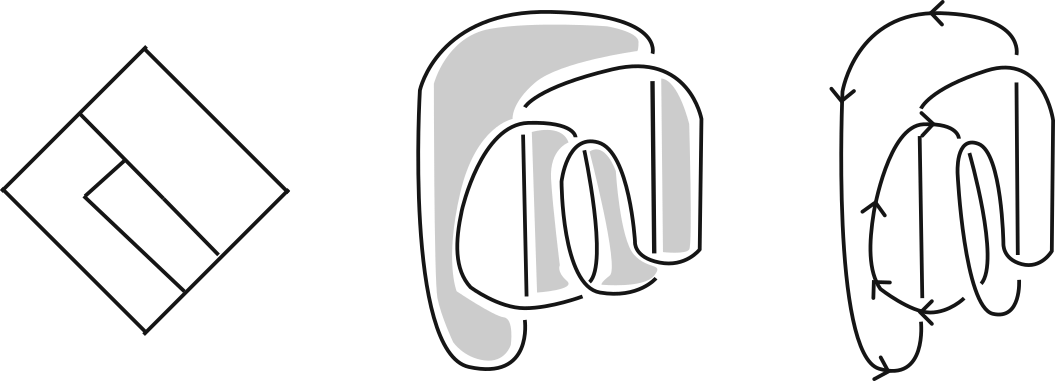}\put(-254,-10){$g \in \vec{F}$}\put(-35,-10){$\vec{\mathcal{L}}(g)$}
{\small 
\put(-166,40){$+$}
\put(-139,40){$-$}
\put(-120,40){$+$}
\put(-104,40){$-$}
}
\caption{An oriented link $\vec{\mathcal{L}}(g)$ built from $g \in \vec{F}.$} \label{fig: checkerboard}
\end{figure}
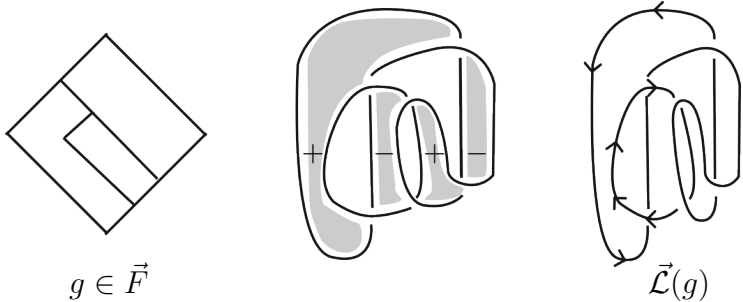
\subsection{Thompson's Group $T$}

$T$ is the group of piecewise-linear orientation-preserving self-homeomorphisms of $S^{1}$, thought of as the unit interval with its endpoints identified, such that derivatives are powers of $2$ and all points of non-differentiability occur at dyadic numbers. $F$ is the subgroup of $T$ whose elements send $1 \mapsto 1$. Elements of $T$ are given by triples $(R,S;k)$ where $(R,S)$ is a pair of planar, rooted, binary trees and $k$ is a positive integer between $1$ and the number of leaves of $R$ and $S$. The integer $k$ indicates that the first part of $R$ is sent to the $k$th part of $S$, and this triple $(R,S;k)$ determines an element of $T$. Observe that $k=1$ if and only if $g=(R,S;k) \in F$. To indicate the value of $k$ in a pair-of-trees diagram, a decoration is placed on the $k$th leaf of $S$, as in Figure \ref{fig: T trees}. 

\begin{figure}[H]
    \includegraphics[scale=0.35]{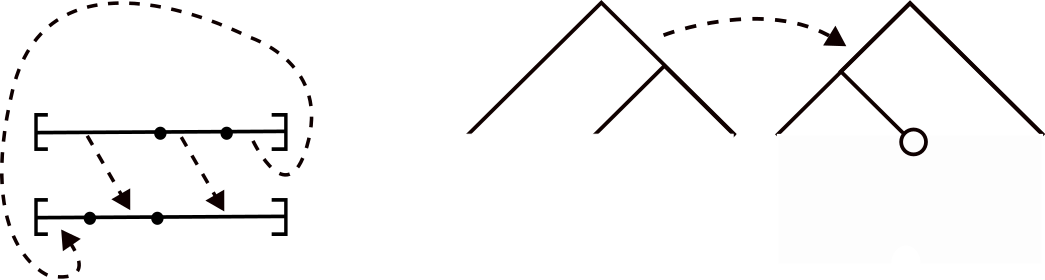}\put(-80,79){$g$}\put(-237,50){$\frac{1}{2}$}\put(-219,50){$\frac{3}{4}$}\put(-254,0){$\frac{1}{4}$}\put(-237,0){$\frac{1}{2}$}
    \caption{The reduced pair of trees and decorated leaf representing the element $g \in T$ which maps $[0,\frac{1}{2}]\mapsto [\frac{1}{4},\frac{1}{2}]$, $[\frac{1}{2},\frac{3}{4}]\mapsto [\frac{1}{2},1]$, and $[\frac{3}{4},1]\mapsto [0,\frac{1}{4}]$. }\label{fig: T trees}
\end{figure}
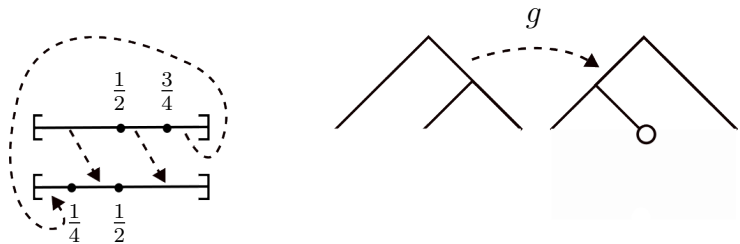

As was the case for $F$, one can refine the partitions $R$ and $S$ to produce an unreduced triple $(R',S';k')$ which differs from $(R,S;k)$ by cancelling carets; see Figure \ref{fig: T trees with caret}. Cancelling carets are slightly less obvious for diagrams in $T \setminus F$ due to the fact that interval corresponding the first leaf of $R$ is not mapped to the interval corresponding to the first leaf of $S$.

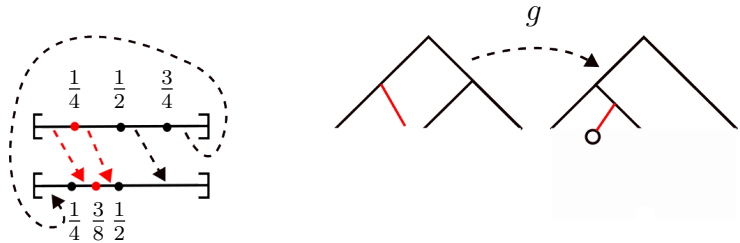
\begin{figure}[H]
    \includegraphics[scale=0.35]{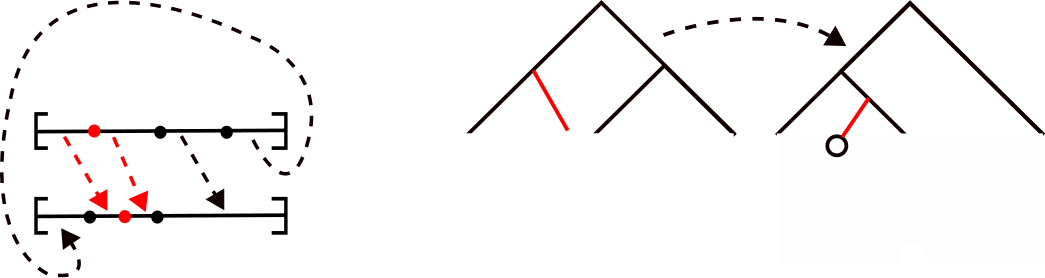}\put(-80,79){$g$}\put(-237,50){$\frac{1}{2}$}\put(-219,50){$\frac{3}{4}$}\put(-254,0){$\frac{1}{4}$}\put(-237,0){$\frac{1}{2}$}\put(-254,50){$\frac{1}{4}$}\put(-245,0){$\frac{3}{8}$}
    \caption{An unreduced triple representing the same element $g$ as in Figure \ref{fig: T trees}, which differs from the triple in Figure \ref{fig: T trees} by a cancelling caret, shown in red.}\label{fig: T trees with caret}
\end{figure}

Section \ref{subsec: FT intro 1} introduced strand diagrams as a way to visualize the group operation in $F$. An analogue for $T$ was developed Belk and Matucci \cite{belk_matucci}. Specifically, every element of $T$ corresponds to a unique reduced \textit{cylindrical strand diagram}, which satisfies the same conditions as a strand diagram, but is now embedded in $S^{1} \times [0,1]$ rather than the unit square\cite{belk_matucci}. Following the definition in \cite{belk_matucci}, isotopic cylindrical strand diagrams are considered equal, and isotopies are not required to fix the boundary circles. Therefore, cylindrical strand diagrams differing by Dehn twists are considered equal.

To associate a cylindrical strand diagram to an element $g=(R,S;k) \in T$, place the trees $R$ and $S$ in the cylinder as in Figure \ref{fig: T strand diag}. Identify leaves such that the first leaf of $R$ is sent to the $k$th leaf of $S$, and then connect the rest of the leaves in unique way for which the graph remains embedded; see Figure \ref{fig: T strand diag}.

\begin{center}
\begin{figure}[H]
\includegraphics[scale=0.55]{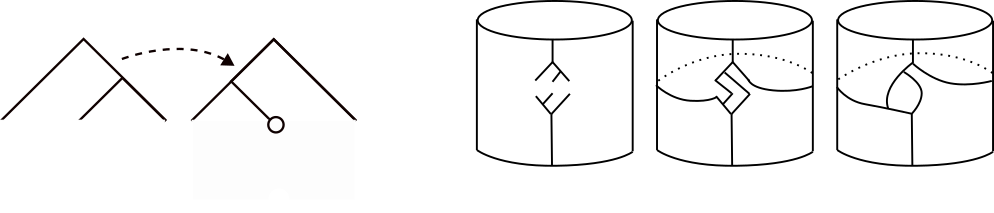}\put(-340,75){$g$}\put(-38,97){$D_{g}$}
\caption{A cylindrical strand diagram $D_{g}$ built from $g \in T$. }\label{fig: T strand diag}
\end{figure}
\end{center}

In Figure \ref{fig: T strand diag}, the rightmost picture differs from the picture to its left by the smoothing of edges. For the rest of this paper strand diagrams built from $T$ will appear without smoothed edges, to indicate the pair of trees from which the diagram was created. 

Cylindrical strand diagrams may be reduced according to local moves of Type I and II as in Figure \ref{fig: comp in F}. As was the case for $F$, given cylindrical strand diagrams for $f, g \in T$, vertically stacking the cylinders and reducing using moves of Type I and II results in the reduced cylindrical strand diagram for $g \circ f$\cite{belk_matucci}. Just as strand diagrams are used to build links from $F$, this paper will use cylindrical strand diagrams to construct annular links from $T$. A forthcoming paper uses strand diagrams to relate Thompson's group F to Khovanov homology of links in 3-space \cite{KLL, khovanov}; it may be possible to use cylindrical strand diagrams to give an analogue for the group $T$ and annular link homologies.

\section{Building Annular Links from Thompson's Group $T$}\label{sec: construction}

To construct annular links from $T$, this section introduces two equivalent methods analogous to those introduced by Jones for $F$.

The first method is pictured in Figure \ref{fig: strand diag construction}. Given $g=(R,S;k)$, consider the associated strand diagram $D_{g}$. Following the method for building links from $F$, add edges to $R$ and $S$ to make them ternary trees, and consider these new edges numbered from left to right. The edge above each root is considered to be numbered $0$. Next, stack ``empty" cylinders above and below $D_{g}$ and connect numbered edges with non-crossing arcs according to the following rule: when the $n$th edge of $R$ is connected to the $m$th edge of $S$ and $n > m$, the arc connecting them must wrap around the annulus. Otherwise, the arc does not wrap around the annulus. Note that this rule guarantees that the arc connecting the top root to another edge will never wrap around the cylinder, and the arc connecting the bottom root to another edge will always wrap around the cylinder, unless a single arc connects the two roots (in which case $g \in F$). Finally, all $4$-valent vertices become crossings as before.

\begin{center}
    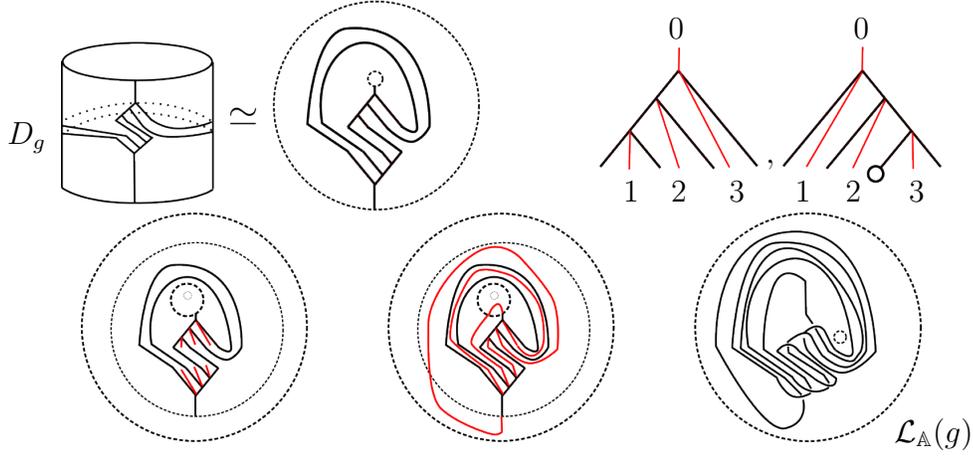
\begin{figure}[H]\includegraphics[scale=0.4]{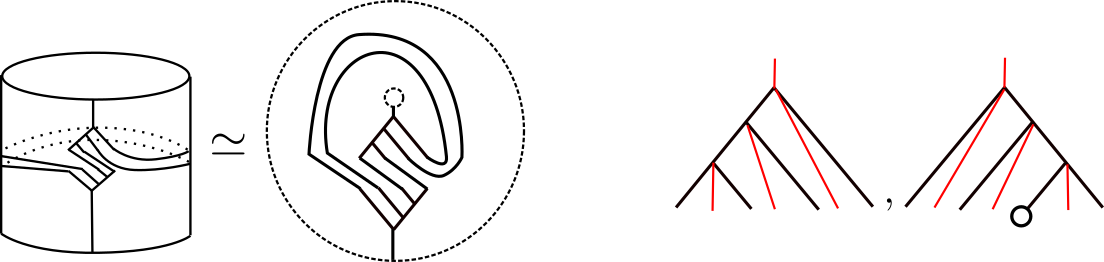}\put(-353,25){$D_{g}$}\put(-103,65){$0$}\put(-33,65){$0$}\put(-120,3){$1$}\put(-102,3){$2$}\put(-80,3){$3$}\put(-55,3){$1$}\put(-36,3){$2$}\put(-12,3){$3$}\\ \includegraphics[scale=0.4]{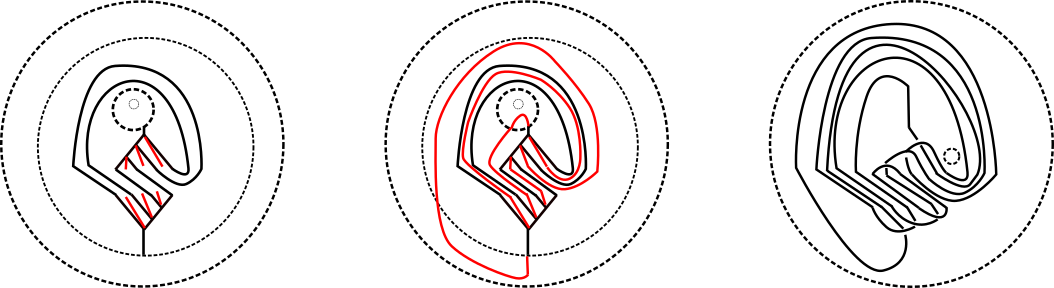}\put(-10,0){$\mathcal{L}_{\mathbb{A}}(g)$}\caption{Building the annular link $\mathcal{L}_{\mathbb{A}}(g)$ from $g \in T$ via the strand diagram $D_{g}$.}\label{fig: strand diag construction}
    \end{figure}
\end{center}

The second method for building annular links from $T$, pictured in Figure \ref{fig: annular tait graph}, involves building an edge-signed graph $\Gamma_{\mathbb{A}}(g) \hookrightarrow \mathbb{A}$ and defining $\mathcal{L}_{\mathbb{A}}(g):=L_{\mathbb{A}}(\Gamma_{\mathbb{A}}(g))$. $\Gamma_{\mathbb{A}}(g)$ is built from $\Gamma(R)$ and $\Gamma(S)$, which are created as in Section \ref{sec: FT intro}. However, the first vertex of 
$\Gamma(R)$ is now identified with the $k$th vertex of $\Gamma(S)$, and edges of $R$ attaching to edges to their left in $S$ must wrap counterclockwise around $\mathbb{A}$; see Figure \ref{fig: annular tait graph}. This second construction is used in Section \ref{sec: proof} to prove Theorem \ref{thm: main}.
\begin{center}
\begin{figure}[h]
\includegraphics[scale=0.6]{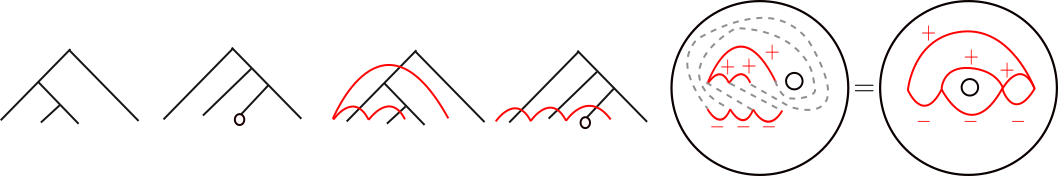}\put(-450,0){$R$}\put(-225,0){$\Gamma(S)$}\put(-300,0){$\Gamma(R)$}\put(-380,0){$S,k$}\put(-70,-15){$\Gamma_{\mathbb{A}}(g) \hookrightarrow \mathbb{A}$}
\caption{Building the graph $\Gamma_{\mathbb{A}}(g)$ from $g=(R,S;k)$. Dotted lines denote identification of vertices; they are not edges.}
\label{fig: annular tait graph}
\end{figure}
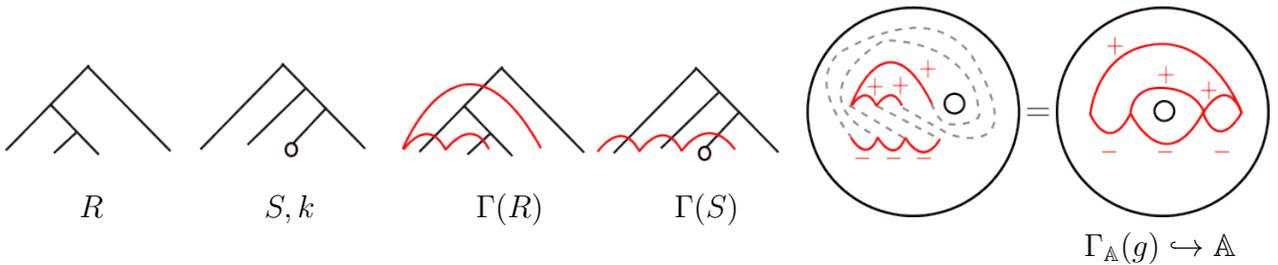
\end{center}

Annular links created from $T$ are closely related to Jones' \textit{planar} links built from $T$ see\cite[Section $4.2$]{jones14}). By construction, $\Gamma(g)$ is the image of $\Gamma_{\mathbb{A}}(g)$ under the inclusion $\mathbb{A} \hookrightarrow \mathbb{R}^{2}$. Consequently, the diagram $\mathcal{L}(g)$ is the image of the diagram $\mathcal{L}_{\mathbb{A}}(g)$ under the same inclusion.
From this we can relate the Kauffman bracket of $\mathcal{L}_{\mathbb{A}}(g)$ to that of $\mathcal{L}(g)$:
\begin{prop}\label{prop: bracket} Let $g \in T$ be given by $(R,S;k)$. Consider $\mathcal{L}_{\mathbb{A}}(g)$ as an element of $\C[x]$, the Skein module $\mathcal{S}(\mathbb{A})$. Evaluating at $x=(-t^{\frac{-1}{2}}-t^{\frac{1}{2}})$ returns the Kauffman Bracket of $\mathcal{L}(g) \in S^{3}$. 
\end{prop} To discuss the analogous result for the Jones polynomial, we must first discuss the oriented subgroup $\vec{T}$, introduced by Jones \cite{jones14}. Defined analogously to $\vec{F}$, \[\vec{T}:=\{g \in T | Chr_{\Gamma(g)}(2)=2\}.\] It follows that for $g \in \vec{T}$, 
$\mathcal{L}_{\mathbb{A}}(g)$ has a natural orientation. The following proposition relates the Jones polynomial of $\mathcal{L}_(g)$ to that of $\mathcal{L}_{\mathbb{A}}(g)$. 
\begin{prop}\label{prop: jones poly} Let $ g \in \vec{T}$. Setting the value of a circle wrapping once around $\mathbb{A}$ equal to that of a trivial circle, the Jones polynomial of $\mathcal{L}_{\mathbb{A}}(g)$ is equal to that of $\mathcal{L}(g)$.
\end{prop}

Propositions \ref{prop: bracket} and \ref{prop: jones poly}, together with Aiello and Conti's proofs of \cite[Theorems $6.2$ and $7.4$]{aielloconti}, imply Corollary \ref{thm: jones rep}.

\subsection{Annular Thompson Badness}\label{subsec: ATB}We now establish an annular analogue for Jones' Thompson Badness. In this section and Section \ref{sec: proof}, we think of $\mathbb{A}$ as $\mathbb{D}^{1}\setminus\{(0,0)\}$. 

\begin{dfn}Let $\Gamma \hookrightarrow \mathbb{A}$ be an edge-signed graph. We say $\Gamma$ is \textbf{$ATB$-friendly} if: 
\begin{itemize}
    \item $\Gamma$ has no loops.
    \item all vertices lie on the $x$ axis.
    \item all edges have interiors either entirely above, or entirely below the $x$ axis.
\end{itemize}  
\end{dfn} 

Now suppose a graph $\Gamma \hookrightarrow \mathbb{A}$ is $ATB$-friendly. Define $e^{in}_{v}, e^{up}, e^{down}, e^{up}_{-}, e^{down}_{+}$ as before. Let $v_{L}$ describe the leftmost vertex and let $v_{1}$ describe the vertex immediately to the right of the origin. 
Label the $N$ vertices by $\{1,\ldots, N\}$ such that $v_{1}$ is labelled $1$, the vertex immediately to its right is labelled $2$, and so on, until the rightmost vertex is labelled $k-1$. Then label the leftmost vertex $k$ and continue increasing left to right until the vertex immediately to the left of the origin is labelled $N$. For example, \\
\begin{center}
\includegraphics[scale=0.3]{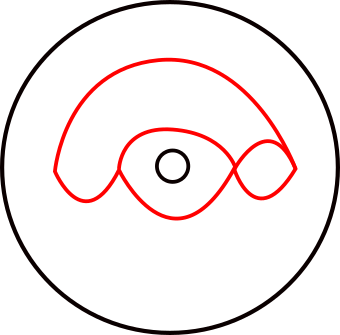}\put(-68,18){$3$}\put(-53,18){$4$}\put(-30,18){$1$}\put(-15,18){$2$}.
\end{center} Let $l(v)$ refer to the label of $v$. By construction $\l(v_{L})=k$.
Define \[e^{<}_{v}:=\{e \in e(\Gamma):  e\text{ connects $v$ to some $w$ such that }l(w) < l(v)\}.\] Now define Annular Thompson Badness, or $ATB$, as follows: \[ATB(\Gamma):=\sum_{v \in V(\Gamma)\setminus v_L}|1-|e^{in}_{v} \cap e^{down}|| + \sum_{v \in V(\Gamma)\setminus v_{1}} |1-|e^{<}_{v}\cap e^{up}||+|e^{up}_{-}|+|e^{down}_{+}|.\]

The following proposition motivates this definition as the correct analogue for Thompson Badness. 
\begin{prop}\label{prop: ATB} Let $\Gamma \hookrightarrow \mathbb{A}$ be an $ATB$-friendly graph. Then $ATB(\Gamma)=0$ if and only if $\Gamma=\Gamma_{\mathbb{A}}(g)$ for some $g \in T$.
\end{prop}
\begin{proof} Suppose $\Gamma= \Gamma_{\mathbb{A}}(g)$, where $g=(R,S;k)$. By construction, $|e^{up}_{-}|=|e^{down}_{+}|=0$. 

It remains to show $\sum_{v \in V(\Gamma)\setminus v_\ell}|1-|e^{in}_{v} \cap e^{down}||=\sum_{v \in V(\Gamma)\setminus v_{1}} |1-|e^{<}_{v}\cap e^{up}||=0.$ Beginning with the first quantity, let $\Gamma_{-}$ denote the subgraph of $\Gamma$ whose edges have interiors in the lower half plane. Since $\Gamma_{-}=\Gamma(S)$ and $\Gamma(R,S)$ is Thompson, $\sum_{v\in V(\Gamma)\setminus v_{L}}|1-|e^{in}_{v}\cap e^{down}||=0$. 

It remains to show that \[ \sum_{v \in V(\Gamma)\setminus v_{1}} |1-|e^{<}_{v}\cap e^{up}||=0.\] Define $\Gamma_{+}$ analogously to $\Gamma_{-}$. Because of the edges wrapping around the annulus, $\Gamma_{+}\neq \Gamma(R)$, but we can recover $\Gamma(R)$ from $\Gamma_{+}$. This is accomplished by embedding $\Gamma_{+}$ in $\R^{2}$ so that labels of the edges increase from left to right. Call this embedding $\Gamma_{+}'$ and observe that $\Gamma_{+}'=\Gamma(R)$.
\begin{center}
\includegraphics[scale=0.4]{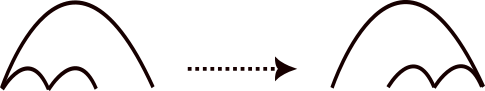}\put(-147,-12){$3$}\put(-135,-12){$4$}\put(-120,-12){$1$}\put(-105,-12){$2$}\put(-50,-12){$1$}\put(-32,-12){$2$}\put(-18,-12){$3$}\put(-5,-12){$4$}\put(-153,22){$\Gamma_{+}$}\put(-55,22){$\Gamma_{+}'$}
\end{center} Letting $v_{L}'$ refer to the leftmost vertex of $\Gamma_{+}'$, we have 

\[\sum_{v\in V(\Gamma_{+}')\setminus \{v_{L}'\}}|1-|e^{up}\cap e^{in}_{v}||=\sum_{v \in V(\Gamma)\setminus v_{1}}|1-|e^{up} \cap e^{<}_{v}|| .\] Since $\Gamma(R,S)$ has Thompson badness zero, the left hand side must be zero. Therefore, $ATB(\Gamma_{\mathbb{A}}(g))=0$. \\ \\
Conversely, let $\Gamma$ be a graph with $ATB(\Gamma)=0$. Let $(R,S)$ be the unique pair of trees corresponding to $(\Gamma_{+}', \Gamma_{-})$. Then $\Gamma=\Gamma_{\mathbb{A}}(g)$ where $g=(R,S;l(v_{L}))$.
\end{proof}

\section{Proof of Theorem \ref{thm: main}}\label{sec: proof}
This section uses Annular Thompson Badness to prove \ref{thm: main}. Given a general graph $\Gamma$, we wish to find a graph $\Gamma'$ such that $L_{\mathbb{A}}(\Gamma)\simeq L_{\mathbb{A}}(\Gamma')$, with $ATB(\Gamma')<ATB(\Gamma)$. For this we use Jones' definition of $2$-equivalence \cite{jones14}. 
 
\subsection{$2$-equivalence} Two edge-signed planar graphs are defined to be $2$-equivalent if they are related by a finite sequence of three moves, which Jones calls $2$-moves. If two graphs $\Gamma,\Gamma'$ are $2$-equivalent then their associated links $L(\Gamma)$ and $L(\Gamma')$ are isotopic. 

The first $2$-move is the addition of a $1$-valent vertex, which corresponds to a Reidemeister move of type I. The remaining two $2$-moves, each corresponding to Reidemeister moves of Type II,  are shown in Figure \ref{fig: 2moves}.
\\
\begin{center}
\begin{figure}[H]
\includegraphics[scale=0.3]{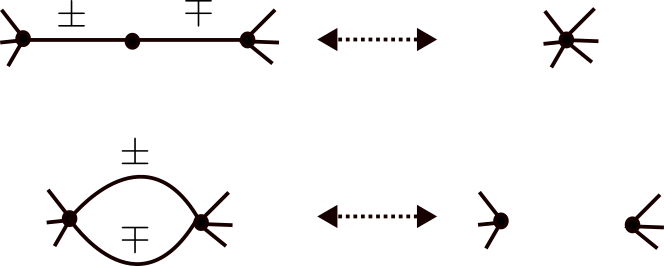}\put(-275, 65){Type IIa}\put(-275,10){Type IIb}
\caption{Moves of Type IIa and IIb as introduced by Jones in \cite{jones14}.}\label{fig: 2moves}
\end{figure}
\end{center}

Jones uses $2$-moves to show that every link has a diagram whose Tait graph has Thompson Badness zero, and thus can be built from an element of the Thompson group. The proof of Theorem \ref{thm: main} will use an analogous strategy to reduce Annular Thompson Badness of a given edge-signed graph $\Gamma \hookrightarrow \mathbb{A}$. To accomplish this we wish to calculate $ATB$ of any graph $\Gamma$, but so far the definition of $ATB(\Gamma)$ requires that $\Gamma$ is $ATB$-friendly. The following lemma takes care of this.

\begin{lemma}\label{lemma:friendly} Let $\Gamma \hookrightarrow \mathbb{A}$ be an edge-signed graph. Then $\Gamma$ is $2$-equivalent to an $ATB$-friendly graph.
    
\end{lemma} 
\begin{proof} Begin by arranging all vertices on the $x$-axis, which is always possible. At this point, if $\Gamma$ is $ATB$-friendly we are done. Otherwise, both loops and edges whose interiors are in both the upper and lower half-plane can be corrected with moves of Type IIa:
\begin{center}
\includegraphics[scale=0.4]{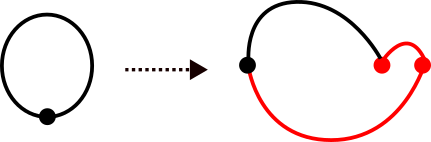}\put(-25, -6){\color{red}$-$}\put(-15,33){\color{red} $+$};\qquad \includegraphics[scale=0.4]{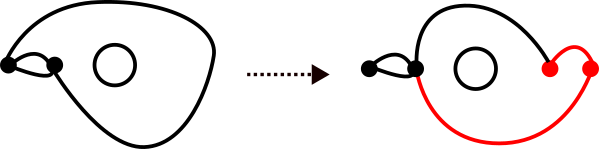}\put(-25, -5){\color{red}$-$}\put(-15,35){\color{red} $+$}.
\end{center}
\end{proof}
Therefore, for any graph $\Gamma \hookrightarrow \mathbb{A}$, we define $ATB(\Gamma):=ATB(\Gamma')$ where $\Gamma'$ is obtained from $\Gamma$  as in the proof of Lemma \ref{lemma:friendly}.

\begin{rmk} The following, when applied inductively, proves Theorem \ref{thm: main}.
\end{rmk}

\begin{theorem} Let $\Gamma \hookrightarrow \mathbb{A}$ be an edge-signed graph. If $ATB(\Gamma)\neq0$, there exists $\Gamma'$ such that $\Gamma'$ is 2-equivalent to $\Gamma$, and $ATB(\Gamma') < ATB(\Gamma)$.
\end{theorem} This proof can be thought of as an extension of Jones' proof of \cite[Lemma $5.3.13$]{jones14} to the annular case. 
\begin{proof} 
To begin, we split into three cases, based on what is causing $ATB(\Gamma) > 0$: \begin{enumerate}
    \item $\sum_{v\in 
 v(\Gamma)\setminus v_{L}}|1-|e^{in}_{v} \cap e^{down}||+\sum_{v \in v(\Gamma)\setminus v_{1}} |1-|e^{<}_v \cap e^{up}||>0$.
 \item The above quantity is zero but  $|e^{up}_{-}|>0$.
 \item The above quantities are zero but $|e^{down}_{+}|>0$.
\end{enumerate} \textbf{Case 1:} The following four-step process will reduce \[\sum_{v\in 
 v(\Gamma)\setminus v_{L}}|1-|e^{in}_{v} \cap e^{down}||+\sum_{v \in v(\Gamma)\setminus v_{1}} |1-|e^{<}_v \cap e^{up}||\] to zero while preserving $2$-equivalence.\\
 \textbf{Case 1, Step 1:} For each $v \in v(\Gamma)\setminus v_{L}$ such that $|e^{in}_v \cap  e^{down}|=0$, let the vertex immediately to the left of $v$ be called $w$ and proceed as in \cite{jones14} regardless of whether $v=v_{1}$:
\\
\begin{center}
\includegraphics[scale=0.5]{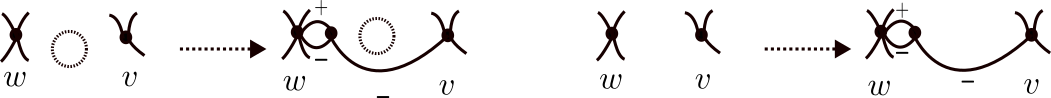}\put(-207,20){or}.
\end{center} The new vertex does not impact $ATB$, and the vertex $v$ now has $|e^{in}_{v} \cap e^{down}|=1$. Each time this step is applied to a relevant vertex, $\sum_{v\in 
 v(\Gamma)\setminus v_{L}}|1-|e^{in}_{v} \cap e^{down}||$ decreases by 1, and all other quantities remain unchanged, so $ATB$ decreases by $1$. \\ \\ \textbf{Case 1, Step 2:} For each $v \in v(\Gamma)\setminus v_{L}$ such that $|e^{in}_{v}\cap e^{down}_{v}|>1$, proceed as in \cite{jones14}:
\begin{center}
\includegraphics[scale=0.4]{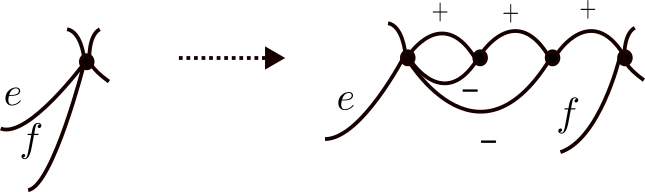}.
\end{center}
Although applying Step $2$ to all relevant vertices will ensure $\sum_{v\in v(\Gamma')\setminus v_{L}}|1-|e^{in}_v \cap e^{down}||=0$, it may increase $\sum_{v \in v(\Gamma')\setminus v_{1}}|1-|e^{<}_{v} \cap e^{up}||$. Take, for example, a vertex $v$ with an outgoing edge stretching over the origin:\\
\begin{center}
    \includegraphics[scale=0.4]{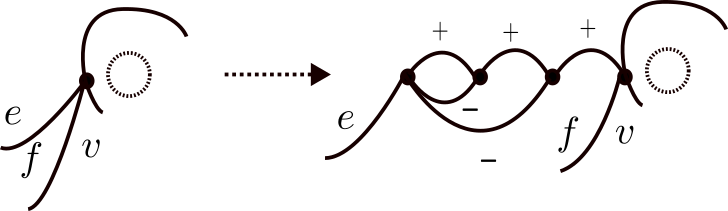}.
\end{center}
In this example, Step $2$ increases $|e^{<}_{v} \cap e^{up}|$ from $1$ to $2$. This will be addressed momentarily in Step $4$, but at this point steps $1$ and $2$ have reduced $\sum_{v \in v(\Gamma)\setminus v_{L}}|1-|e^{in}_{v} \cap e^{down}||$ to $0$. If we also have that $\sum_{v \in v(\Gamma)\setminus v_{1}}|1-|e^{<}_{v} \cap e^{up}||=0$, we are done. Otherwise proceed to step $3$.  \\ \\
\textbf{Case 1, Step 3:} We wish to deal with vertices $v \in V(\Gamma) \setminus v_{1}$ for which $|e^{<}_{v} \cap e^{up}_{v}|=0$. If $v \neq v_{L}$, let $w$ refer to the vertex immediately to the left of $v$ and proceed as in \cite{jones14}:\\
\begin{center}
\includegraphics[scale=0.4]{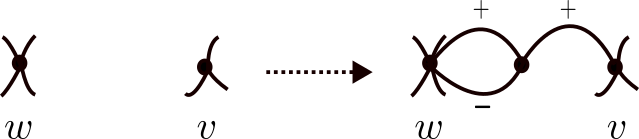}
\end{center} If $v=v_{L}$, modify the graph as follows: \\
\begin{center}
\includegraphics[scale=0.4]{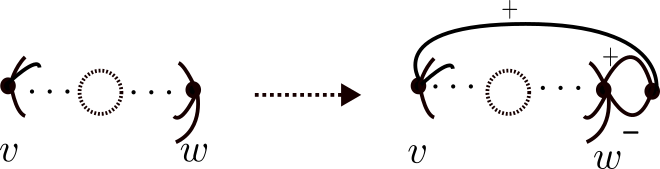}
\end{center} In both cases, $\sum_{v \in v(\Gamma)\setminus v_{1}}|1-|e^{<}_{v} \cap e^{up}||$ decreases by $1$ and all other quantities remain unchanged. Therefore each time this step is applied to a relevant vertex, $ATB$ decreases by $1$.

\textbf{Case 1, Step 4:} We wish to deal with vertices $v \in v(\Gamma) \setminus v_{1}$ for which 
$|e^{<}_{v} \cap e^{up}|>1$. This can happen one of four ways. In any case, proceed as in \cite{jones14}:\\
\begin{center}
    \includegraphics[scale=0.6]{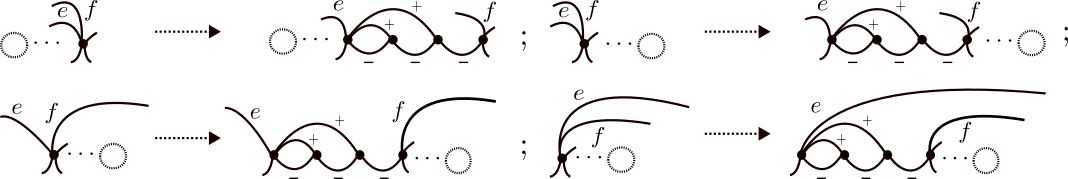}.
\end{center}In each of the four cases, $\sum_{v \in v(\Gamma)\setminus v_{1}}|1-|e^{<}_{v} \cap e^{up}||$ decreases by $1$ and all other quantities remain unchanged. Therefore each time this step is applied to a relevant vertex, $ATB$ decreases by $1$.
After these four steps, we have $\sum_{v\in v(\Gamma)\setminus v_{L}}|1-|e^{in}_{v} \cap e^{down}||+\sum_{v\in v(\Gamma)\setminus v_{1}} |1-|e^{<}_v \cap e^{up}||=0$, and $|e^{up}_{-}| + |e^{down}_{+}|$ remains unchanged. Therefore $ATB$ has been reduced, and this concludes Case $1$. \\ \\
\textbf{Case 2:} We have $\sum_{v\in v(\Gamma)\setminus v_{L}}|1-|e^{in}_{v} \cap e^{down}||+\sum_{v\in v(\Gamma)\setminus v_{1}} |1-|e^{<}_v \cap e^{up}||=0$ and $|e^{down}_{+}|>0$. From now on, only edges in $|e^{up}_{-}|\cup |e^{down}_{+}|$ will be pictured with a sign; the rest are understood to be in $|e^{down}_{-}| \cup |e^{up}_{+}|$. Fix an edge in $e' \in e^{down}_{+}$ and call its terminal vertex $v$. We further split into two cases, depending on whether $v=v_{1}$. Let $w$ refer to the vertex immediately to the left of $v$. If $v \neq v_{1}$, proceed as in \cite{jones14}:\\
\begin{center}
\begin{figure}[H]
\includegraphics[scale=0.4]{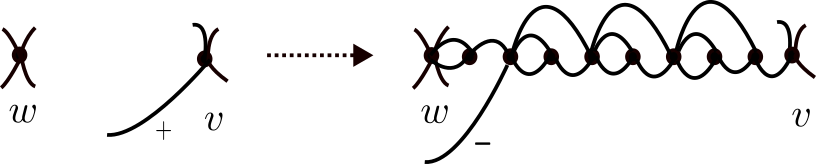}
\caption{}\label{fig: edown fix}
\end{figure}
\end{center}
If $v=v_{1}$, proceed as follows:  \\
\begin{center}
\begin{figure}[H]
\includegraphics[scale=0.4]{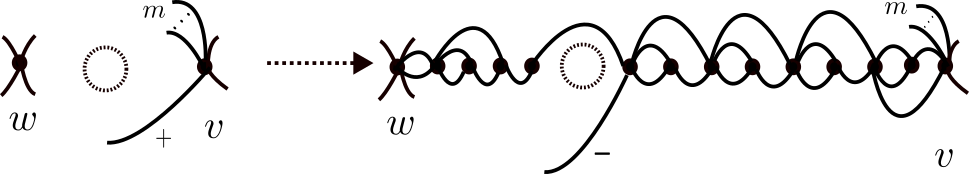}.
\caption{}
\label{fig: extra edges}
\end{figure}
\end{center} Note that $m$ denotes the number of edges coming into $v$ from the other side of the annulus. Distinguishing these edges in the picture is necessary because, by virtue of stretching over the origin, they are not in $e^{<}_{v}$ and will not affect $ATB$. 

One may wonder why we treat $v=v_{1}$ differently from $v\neq v_{1}$. The following picture depicts what would have happened if we did not. \\
\begin{center}
    \includegraphics[scale=0.4]{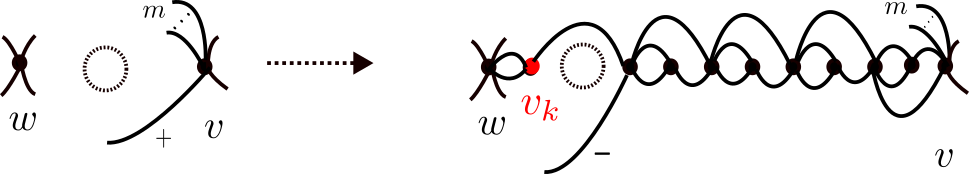}
\end{center} The vertex $v_{k}$ now has $|e^{<}_{v_{k}} \cap e^{\text{up}}|=2$ which increases $ATB$ by $1$ and necessitates a correction as in Case 1, step 4, bringing us back to Figure \ref{fig: extra edges}.

To see how Figure \ref{fig: edown fix} preserves $2$-equivalence, use Type IIb moves to remove cancelling $2$-cycles with opposite signs, then use Type I moves to eliminate 1-valent vertices, and finally use Type IIa moves to collapse cancelling edges, pictured in red, so that the original edge, pictured in blue, remains. A similar sequence of moves demonstrates the $2$-equivalence for Figure \ref{fig: extra edges}.\\
\begin{center}
\includegraphics[scale=0.55]{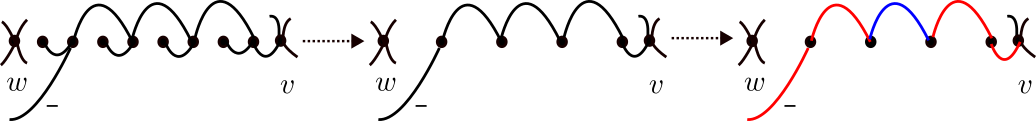}
\end{center} 

The modifications in Figures \ref{fig: edown fix} and \ref{fig: extra edges} decrease $|e^{down}_{+}|$ by $1$ and do not impact the quantity $\sum_{v\in 
v(\Gamma)\setminus v_{L}}|1-|e^{in}_{v} \cap e^{down}||+\sum_{v \in v(\Gamma)\setminus v_{1}} |1-|e^{<}_v \cap e^{up}||$ or $|e^{up}_{-}|$. Therefore each time this modification is applied to an edge in $e^{down}_{+}$, $ATB$ decreases by $1$ and any edge in $e^{down}_{+}$ can be corrected.

\textbf{Case 3}: All other quantities relevant to $ATB(\Gamma)$ are zero, but $|e^{up}_{-}|>0$. Once again we further split into cases, depending on which side of $\mathbb{A}$ contains the terminal vertex $v$ of a problematic edge $e'$.

If $v$ is on the left side of $\mathbb{A}$, let $w$ refer to the vertex immediately to the left of $v$ and proceed as in \cite{jones14}:
\begin{center}
\includegraphics[scale=0.4]{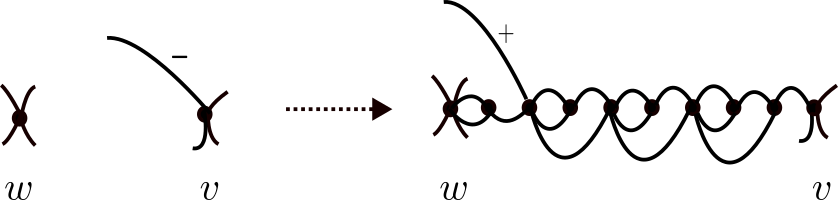}.
\end{center} A key fact that makes the above work is that as long as $v$ is on the left side of $\mathbb{A}$ , we have $|e^{in}_{v} \cap e^{up}|=1$, which was assumed in \cite{jones14}. When this is not true, a different correction will be required; specifically, if instead $e'$ terminates at some $v$ on the right side of $\mathbb{A}$, we may have that $|e^{in}_{v} \cap e^{up}|>1$, due to any number of edges entering $v$ from above which stretch over the origin. We must further divide into two cases, based on whether $e'$ itself stretches over the origin.

If not, perform the following combination of moves of Type I, IIa, and IIb, finitely many times until $|e^{in}_{v} \cap e^{up}|=1$: 
\begin{center}
\includegraphics[scale=0.4]{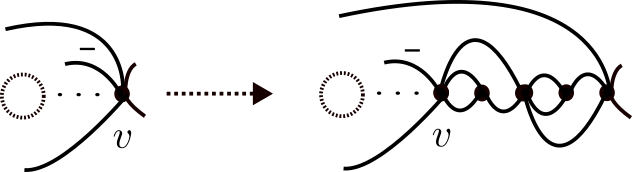}.
\end{center} Then, we may proceed as in the previous case.
If, on the other hand, $e'$ does stretch over the origin, apply the above modification to isolate $e'$ from all other edges in $|e^{in}_{v} \cap e^{up}|$ stretching over the origin. Once the problematic edge is isolated, one may modify the graph as follows,
\begin{center}
\includegraphics[scale=0.4]{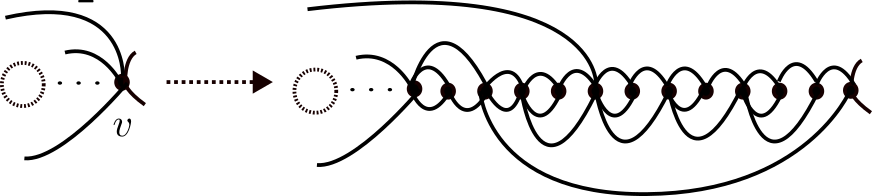}.
\end{center} To see $2$-equivalence, use Type IIb moves to remove cancelling $2$-cycles, then use Type I moves to remove $1$-valent vertices. Lastly apply Type IIa moves to collapse the three pairs of cancelling edges pictured below in red, blue, and green.
\begin{center}
\includegraphics[scale=0.5]{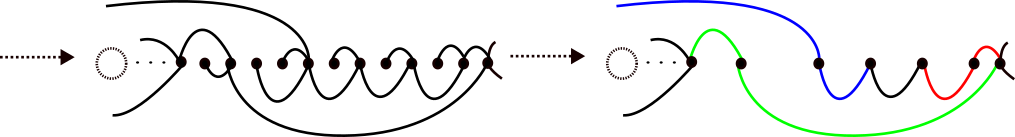}
\end{center} 

The modifications in this step reduce $|e^{up}_{-}|$ by $1$ and do not affect any other quantities relevant to $ATB$. This concludes Case 3. These three cases demonstrate that any edge-signed graph embedded in $\mathbb{A}$ with nonzero $ATB$ is $2$- equivalent to a graph with lower $ATB$.
\end{proof}

\subsection{Example: a positive trefoil embedded in $\mathbb{A}$} %can skip all the stuff in the middle and just go straigh to the pair of trees, maybe even try to make them all on the same line. Also add a little circle for k. 

Let $\Gamma \hookrightarrow \mathbb{A}$ be the graph pictured below, which corresponds to the positive trefoil embedded as in Figure \ref{fig: tait graph ex}. The graph $\Gamma \hookrightarrow \mathbb{A}$ has $|e^{down}_{+}|=1$, $|e^{up}_{-}|=0$, $|e^{in}_{v_{3}} \cap e^{down}|=0$, $|e^{in}_{v_{1}} \cap e^{down}|=1$, $|e^{<}_{v_{2}} \cap e^{up}|=0, |e^{<}_{v_{3}} \cap e^{up}|=2$. Following the process outlined in Theorem \ref{thm: main} leads to the element $g=(R,S;k) \in T$ pictured below, and $\mathcal{L}_{\mathbb{A}}(g)=L_{\mathbb{A}}(\Gamma)$.\\
\begin{center} 
\begin{figure}[H]%make this look nicer haha 
\includegraphics[scale=0.5]{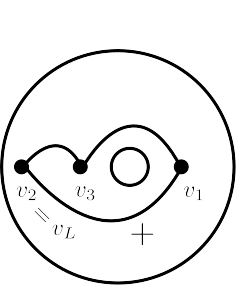} \put(-66,96){$\Gamma \hookrightarrow \mathbb{A}$}
\hspace{0.4in} 
\includegraphics[scale=0.5]{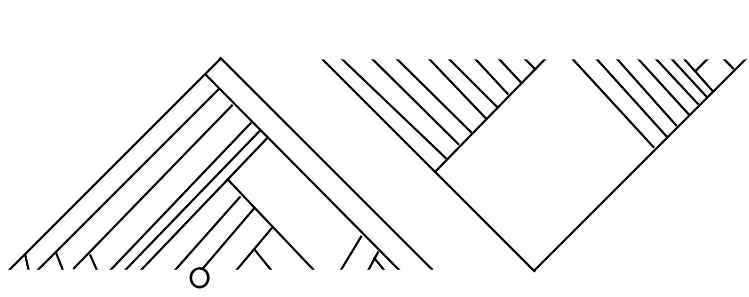}\put(-82,100){$S$}\put(-205,100){$R$}
\caption{The triple $(R,S;k)$ such that $\mathcal{L}_{\mathbb{A}}(R,S;k)=L_{\mathbb{A}}(\Gamma).$}
%\put(-220,90){$g(I)=J$} 
\end{figure}
\end{center}

\bibliographystyle{amsplain}
\bibliography{biblio.bib}

\end{document}